\documentclass{amsart}
\usepackage[utf8]{inputenc}
\usepackage{amsmath, amssymb, amsthm}
\usepackage{mathrsfs}
\usepackage{enumerate}

\newtheorem{thm}{Theorem}[section]

\newtheorem{cor}[thm]{Corollary}
\newtheorem{lem}[thm]{Lemma}
\theoremstyle{definition}
\newtheorem{defi}[thm]{Definition}
\newtheorem{claim}{Claim}[thm]

\newtheorem{case}{Case}

\newcommand{\N}{\mathbb{N}}
\newcommand{\R}{\mathbb{R}}
\newcommand{\X}{\mathcal{X}}
\newcommand{\U}{\mathcal{U}}
\newcommand{\V}{\mathcal{V}}
\newcommand{\W}{\mathcal{W}}
\newcommand{\fin}{\mathrm{FIN}}
\newcommand{\varep}{\varepsilon}
\DeclareMathOperator{\range}{range}

\DeclareMathOperator{\supp}{supp}
\DeclareMathOperator{\stem}{stem}

\begin{document}

\title[Parametrized Ramsey theory of infinite block sequences]{Parametrized Ramsey theory of infinite block sequences of vectors}
\date{\today}

\author[J. K. Kawach]{Jamal K. Kawach}
\thanks{Research partially supported by an Ontario Graduate Scholarship.}
\address{Department of Mathematics\\ University of Toronto\\ Toronto, Canada, M5S 2E4.}
\email{jamal.kawach@mail.utoronto.ca}
\urladdr{https://www.math.toronto.edu/jkawach}


\begin{abstract}
We show that the infinite-dimensional versions of Gowers' $\fin_k$ and $\fin_{\pm k}$ theorems can be parametrized by an infinite sequence of perfect subsets of $2^\omega$. To do so, we use ultra-Ramsey theory to obtain exact and approximate versions of a result which combines elements from both Gowers' theorems and the Hales-Jewett theorem. As a consequence, we obtain a parametrized version of Gowers' $c_0$ theorem.
\end{abstract}

\maketitle

\section{Introduction}
Recall that Hindman's theorem \cite{H} says that for any finite colouring of the set $\fin$ of finite subsets of $\omega$, there is an infinite block sequence $B$ such that the set of all finite unions of elements of $B$ is monochromatic, where a sequence $(p_n)_{n<\omega}$ of finite subsets of $\omega$ is a \emph{block sequence} if $$\max p_n < \min p_m \text{ whenever $n < m$}.$$ This was then generalized by Gowers \cite{G} in order to prove that every real-valued Lipschitz (or, more generally, uniformly continuous) function $f$ on the unit sphere of $c_0$ is \emph{oscillation stable}: For every $\varep > 0$ there is an infinite-dimensional subspace $X$ of $c_0$ such that the oscillation of $f$ is at most $\varep$ when restricted to the unit sphere of $X$. In fact, the proof of Gowers' $c_0$ theorem is essentially combinatorial and makes use of an approximate Ramsey theorem concerning $\fin_{\pm k}$, the set of all finitely-supported functions $p : \omega \rightarrow \{0, \pm 1, \dots, \pm k\}$ such that $p$ achieves at least one of the values $\pm k$. In this setting, elements of $\fin_{\pm k}$ can be naturally identified with vectors in $c_0$. An exact version of such a result exists for $\fin_k$, the set of all finitely-supported functions $p : \omega \rightarrow \{0, 1,\dots,k\}$ such that $k \in \range(p)$; this can be used to prove an oscillation stability result for Lipschitz functions on the positive part of the unit sphere of $c_0$.

Hindman's theorem was extended in another direction by Milliken \cite{M} who proved an \emph{infinite-dimensional} version of Hindman's theorem: For every analytic colouring of the set $\fin^{[\infty]}$ of all infinite block sequences of finite subsets of $\omega$, there is an infinite block sequence $B$ such that the set $$\{C \in \fin^{[\infty]} : \text{ every $X \in C$ is a union of sets from $B$} \}$$ is monochromatic. Given such an infinite-dimensional result, one natural way to strengthen it is to \emph{parametrize} it by some space of interest. One line of research in this direction is to parametrize such results by perfect subsets of the Cantor space $2^\omega$ with its standard metrizable topology, in the sense that we colour a product of the form $X \times 2^\omega$ and look for monochromatic subsets of the form $Y \times P$ where $Y$ is a ``nice'' subset of $X$ and $P$ is a perfect subset of $2^\omega$. The first result of this kind is due to Miller and Todorcevic \cite[p. 183]{Miller} and involves a parametrization of the Galvin-Prikry theorem \cite{GP}; Pawlikowski \cite{Paw} later showed that Ellentuck's theorem \cite{E} can be parametrized by perfects subsets of $2^\omega$. The result of Pawlikowski was then generalized by Mijares and Nieto \cite{Mi, MN} who eventually proved that the abstract Ramsey theorem of Todorcevic \cite{T} can be parametrized as above.

Instead of asking for a parametrization involving one perfect subset of $2^\omega$, one can look for a \emph{sequence} of perfect subsets of $2^\omega$. For instance, Milliken's theorem was parametrized by Todorcevic \cite[Theorem 5.45]{T} using sequences of perfect subsets of $2^\omega$ in the following way; the proof makes use of an infinite-dimensional version of the Hales-Jewett theorem \cite[Corollary 5.42]{T}.

\begin{thm}[Parametrized Milliken Theorem]
For every finite Souslin measurable colouring of $\fin \times (2^\omega)^\omega$ there are $B \in \fin^{[\infty]}$ and a sequence $(P_i)_{i<\omega}$ of non-empty perfect subsets of $2^\omega$ such that $[B]^{[\infty]} \times \prod_{i<\omega} P_i$ is monochromatic.
\end{thm}

More generally, one can ask which spaces admit a parametrization using sequences of perfect subsets of $2^\omega$. In her thesis, Zheng \cite{Zheng} isolated a necessary and sufficient condition for the existence of such a parametrization within the context of topological Ramsey space theory; we refer the reader there for more information and for applications of parametrized Ramsey theory.

The main goal of this paper is to show that the Parametrized Milliken Theorem holds when $\fin$ is replaced with $\fin_k$. On the other hand, while an exact Ramsey theorem is not possible in the setting of $\fin_{\pm k}$, we will obtain an ``approximate'' parametrized Ramsey theorem for $\fin_{\pm k}$. The proofs make use of \emph{ultra-Ramsey theory}; this approach is advantageous since it allows for more flexibility when dealing with ``approximate'' Ramsey-theoretic results.

The rest of this paper is organized as follows. In Section 2 we give a brief overview of Gowers' theorems as well as their infinite-dimensional counterparts. In Section 3 we use ultra-Ramsey theory to obtain an infinite-dimensional Ramsey theorem which can be seen as a common generalization of Gowers' $\fin_k$ theorem and a multi-variable version of the Hales-Jewett theorem. This Gowers-Hales-Jewett theorem is then used to parametrize the infinite-dimensional $\fin_k$ theorem. Section 4 contains the approximate versions of the results from Section 3; in particular we use ultrafilter methods to obtain an approximate Gowers-Hales-Jewett theorem relative to a metric defined on multi-variable words and then use this to parametrize the infinite-dimensional $\fin_{\pm k}$ theorem. We then conclude with an application to oscillation stability of functions on the unit sphere of $c_0$.

\section{Preliminaries}

Let $\omega$ denote the set of all non-negative integers and let $\N = \omega \setminus \{0\}$. We follow standard set-theoretic conventions. In particular, each ordinal $m < \omega$ will be identified with the set $\{0, \dots, m-1\}$ of its predecessors.

\subsection{Gowers' theorems}

Given $k \in \N$, let $\fin_{\pm k}$ denote the set of all functions $p: \omega \rightarrow \{0, \pm 1, \dots, \pm k\}$ such that $$\supp p := \{n < \omega : p(n) \neq 0\}$$ is finite and such that $p$ achieves at least one of the values $\pm k$. Given $p,q \in \fin_{\pm k}$, write $p < q$ whenever $\max \supp p < \min \supp q$. In this case $p+ q$ will denote the element of $\fin_{\pm k}$ given by the coordinate-wise sum of $p$ and $q$. This operation gives $\fin_{\pm k}$ the structure of a partial semigroup.

We also have an operation between various $\fin$ spaces: The \emph{tetris operation} $T : \fin_{\pm k} \rightarrow \fin_{\pm (k - 1)}$ is defined by
\[ T(p)(n) := \begin{cases}
	p(n) - 1 & \text{ if $p(n) > 0$},\\
	0 & \text{ if $p(n) = 0$},\\
	p(n) + 1 & \text{ if $p(n) < 0$}.
	\end{cases}
\]
It is easy to check that $T$ is a surjective homomorphism of partial semigroups. For $\alpha \leq \omega$, a sequence $(p_n)_{n < \alpha}$ is a \emph{block sequence in $\fin_{\pm k}$} if $p_n \in \fin_{\pm k}$ and $p_n < p_m$ for all $n < m < \alpha$. Let $\fin_{\pm k}^{[\infty]}$ denote the space of all infinite block sequences in $\fin_k$. Given a block sequence $P = (p_n)_{n < \alpha}$, the \emph{partial subsemigroup of $\fin_{\pm k}$ generated by $P$} is defined as
\begin{equation*}
\begin{split}
[P]_{\pm k} := & \{\varep_0 T^{j_0} (p_{n_0}) + \dots + \varep_m T^{j_m} (p_{n_m}) : m < \omega, n_0 < \dots < n_m < \alpha, \\
& \varep_0, \dots, \varep_m \in \{\pm 1\},  j_0, \dots, j_m < k \text{ and $\min j_i = 0$}\}.
\end{split}
\end{equation*}
If $Q = (q_n)_{n < \beta}$, $\beta \leq \alpha$ is another block sequence, write $Q \leq P$ and say $Q$ is a \emph{block subsequence of $P$} whenever $q_n \in [P]_{\pm k}$ for all $n < \beta$. We write $[P]_{\pm k}^{[\infty]}$ for the set of all infinite block subsequences of $P$.

For a subset $A \subseteq \fin_{\pm k}$ and $\varep > 0$, define $$(A)_\varep := \{p \in \fin_{\pm k} : (\exists q\in A) \, ||p - q||_\infty \leq \varep\}$$ where $||\cdot||_\infty$ denotes the $\ell_\infty$ norm. We can now state the following theorem of Gowers, originally proved in \cite{G} using the theory of idempotent ultrafilters in order to show that every real-valued uniformly continuous function on the unit sphere of $c_0$ is oscillation stable.

\begin{thm}[Gowers]\label{gowersthm}
For every $k, r \in \N$ and every $c: \fin_{\pm k} \rightarrow r$ there are $i < r$ and $P \in \fin_{\pm k}^{[\infty]}$ such that $$ [P]_{\pm k}\subseteq \left(c^{-1}\{i\}\right)_1.$$
\end{thm}

There is also an exact version of Gowers' theorem, which we now describe: Given $k \in \N$, let $\fin_k$ denote the set of all functions $p: \omega \rightarrow k+1$ such that $\supp p$ is finite and $k \in \range(p)$. The ordering $<$ on $\fin_k$ and the sum $p+q$ of two elements of $\fin_k$ are defined analogously. The corresponding tetris operation $T : \fin_k \rightarrow \fin_{k - 1}$ is defined by
\[ T(p)(n) := \begin{cases}
	p(n) - 1 & \text{ if $p(n) > 0$},\\
	0 & \text{ if $p(n) = 0$}.
	\end{cases}
\]

As before, a block sequence in $\fin_k$ is a sequence $P = (p_n)_{n<\omega}$ such that $p_n < p_m$ whenever $n < m$. $\fin_k^{[\infty]}$ will denote the space of all infinite block sequences in $\fin_k$. The partial subsemigroup of $\fin_k$ generated by $P = (p_n)_{n <\omega} \in \fin_k^{[\infty]}$ is
\begin{equation*}
\begin{split}
[P]_k :=  \{ T^{j_0} (p_{n_0}) + \dots +  T^{j_m} (p_{n_m}) : &\,  m < \omega, n_0 < \dots < n_m < \alpha, \\
&   j_0, \dots, j_m < k \text{ and $\min j_i = 0$}\}
\end{split}
\end{equation*} 
and the set of all infinite block subsequences of $P$ will be denoted $[P]_k^{[\infty]}$. The following result was also proved by Gowers in \cite{G}.

\begin{thm}[Gowers]\label{gowersthm1}
For every $k, r \in \N$ and every $c: \fin_k \rightarrow r$ there is $P \in \fin_k^{[\infty]}$ such that $[P]_k$ is monochromatic.
\end{thm}

We refer the reader to \cite{AT,K,T} for details and proofs of Gowers' theorems. The interested reader is also referred to \cite{Ojeda, Tyros} for discussions and proofs of the finite versions of Gowers' theorems.

\subsection{Infinite-dimensional Ramsey theory}

In this note we will be concerned with various infinite-dimensional versions of Gowers' theorems. In this setting one needs a topological restriction on the permitted colourings in order to obtain a Ramsey theorem. To describe such a restriction, first recall that a \emph{Souslin scheme} is a family of subsets $(X_s)_{s \in \omega^{< \omega}}$ of some underlying set which is indexed by finite sequences of non-negative integers. The \emph{Souslin operation} turns a Souslin scheme $(X_s)_{s \in \omega^{< \omega}}$ into the set $$\bigcup_{x \in \mathcal{N}} \bigcap_{n < \omega} X_{x \restriction n}$$ where $\mathcal{N}$ denote the Baire space, i.e. the set of all infinite sequences in $\omega$. Then a subset of a topological space $\mathcal{T}$ is \emph{Souslin measurable} if it belongs to the minimal field of subsets of $\mathcal{T}$ which contains all open sets and is closed under the Souslin operation. In particular, every analytic or coanalytic subset of $\mathcal{T}$ is Souslin measurable when $\mathcal{T}$ is Polish. Finally, we say that a finite colouring $c : \mathcal{T} \rightarrow n$ is Souslin measurable if each fibre $c^{-1}\{i\}, \, i < n$ is Souslin measurable.

The case $k=1$ of the following theorem is due to Milliken \cite{M} while the general case is due to Todorcevic \cite{T}; they are the infinite-dimensional versions of Hindman's theorem \cite{H} and Gowers' $\fin_k$ theorem, respectively.

\begin{thm}[Milliken-Todorcevic Theorem]\label{inffin}
For every finite Souslin measurable colouring of $\fin_k^{[\infty]}$ there is $B \in \fin_k^{[\infty]}$ such that $[B]_k^{[\infty]}$ is monochromatic.
\end{thm}

Similarly, the following result from \cite{JK} gives an infinite-dimensional version of Gowers' $\fin_{\pm k}$ theorem.

\begin{thm}
For every finite Souslin measurable colouring of $c : \fin_{\pm k}^{[\infty]} \rightarrow r$ there are $B \in \fin_{\pm k}^{[\infty]}$ and $i < r$ such that the following holds: For every $A = (a_n)_{n<\omega} \in [B]_{\pm k}^{[\infty]}$ there is $\widetilde{A} = (\widetilde{a}_n)_{n<\omega} \in \fin_{\pm k}^{[\infty]}$ such that $$c(\widetilde{A}) = i \text{ and $||a_n - \widetilde{a}_n||_\infty \leq 1$ for all $n$}.$$
\end{thm}

We take this opportunity to remark that our approach toward obtaining ``approximate'' Ramsey theorems is certainly not optimal; in fact, there are simpler ways of transferring results about $\fin_k$ to $\fin_{\pm k}$ (see, for instance, \cite{Ojeda} for one such method). However, our approach has the advantage of isolating useful ultrafilters which may be of independent interest and which, in some sense, provide an explanation for the existence of the corresponding Ramsey results without relying on the companion result for $\fin_k$. Our hope is that the methods used in this paper (and the related paper \cite{JK}, both of which are in turn based on \cite{T}) will eventually lead to a unified approach to obtaining approximate Ramsey results, even in cases where there is no naturally associated ``exact'' result or when there is no way to transfer the exact results to the approximate setting.

\section{A parametrized Milliken-Todorcevic theorem}

In this section we will show that the parametrized Milliken theorem still holds when $\fin$ is replaced with $\fin_k$:

\begin{thm}[Parametrized Milliken-Todorcevic Theorem]\label{parafin}
For every finite\, \,\, Souslin measurable colouring of $\fin_k^{[\infty]} \times (2^\omega)^\omega$ there are $B \in \fin_k^{[\infty]}$ and a sequence $(P_i)_{i<\omega}$ of non-empty perfect subsets of $2^\omega$ such that $[B]_k^{[\infty]} \times \prod_{i<\omega} P_i$ is monochromatic.
\end{thm}

To do so, we will need a combinatorial result which can be seen a common infinite-dimensional generalization of Gowers' $\fin_k$ theorem and the Hales-Jewett theorem. We remark here that a general framework for obtaining infinitary Gowers-Hales-Jewett theorems has been developed in \cite{L}; other versions are considered in \cite{AvT} and alluded to in \cite{K}. Our approach is heavily inspired by that of \cite{T}.

Throughout this section, fix an infinite alphabet $L = \bigcup_{n < \omega} L_n$ given as an increasing union of finite subalphabets $L_n$, as well as a distinguished letter $0 \in L_0$ together with $k$ distinct variables $v_1, \dots, v_k \not \in L$. $W_L$ will denote the set of all variable-free words over $L$ and, for each $i \in \{1, \dots, k\}$, $W_{Lv_i}$ will denote the set of all variable words $x$ over $L$ such that $$ i = \max\{j \leq k : \text{ $v_j$ appears in $x$}\}.$$ An element of $W_{Lv_k}$ will be called a \emph{$v_i$-variable word}.

Let $$S = W_L \cup \bigcup_{1\leq i \leq k} W_{Lv_i}$$ and work in the semigroup $(S, \, ^\frown)$ where $^\frown$ denotes the concatenation operator on pairs of words. Given $x \in W_{Lv_k}$ and a $k$-tuple $\vec\lambda = (\lambda_1, \dots, \lambda_k) \in L^k \cup \{\vec v\}$, let $x[\vec\lambda]$ be the word obtained by replacing each occurrence of $v_i$ with $\lambda_i$, where $\vec v = (v_1, \dots, v_k)$. In addition to substitution, we also have a version of the tetris operation defined for $v_k$-variable words: Given $x \in W_{Lv_k}$, define $T(x) \in W_{Lv_{k-1}}$ by
\[ T(x)(n) := \begin{cases}
	v_{i-1} & \text{ if $x(n) = v_i$ for $i>1$},\\
	0 & \text{ if $x(n) = v_1$},\\
	x(n) & \text{ if $x(n) \in L$}.
	\end{cases}
\]
Define $T(w) = w$ for each $w \in W_L$. Given a sequence $X = (x_n)_{n < \omega}$ of $v_k$-variable words, the \emph{partial subsemigroup of $W_{Lv_k}$ generated by $X$}, denoted by $[X]_{Lv_k}$, is defined to be the set of all $v_k$-variable words of the form $$T^{j_0}(x_{n_0}[\vec\lambda_0])^\frown \dots ^\frown T^{j_l}(x_{n_l}[\vec\lambda_l])$$ where $l < \omega$, $n_0 < \dots < n_l < \omega$, $j_0, \dots, j_l \leq k$ and $\vec\lambda_i \in L_{n_i}^k \cup \{\vec v\}$ for each $i \leq l$; note that for such an expression to be a $v_k$-variable word, there must be some $i \leq l$ such that $j_i = 0$ and $\vec \lambda_i = \vec v$. We also consider the \emph{partial subsemigroup of $W_L$ generated by $X$}, defined as $$[X]_L = \{ x_{n_0}[\vec\lambda_0]^\frown \dots ^\frown x_{n_l}[\vec\lambda_l] \in W_L : n_0 < \dots < n_l < \omega \text{ and } (\forall i \leq l) \, \vec\lambda_i \in L_{n_i}^k \}.$$

Let $W_{Lv_k}^{[\infty]}$ denote the set of all infinite sequences $X = (x_n)_{n<\omega}$ in $W_{Lv_k}$ which are \emph{rapidly increasing}, i.e. sequences $(x_n)$ such that $$|x_n| > \sum_{i<n} |x_i| \text{ for all $n< \omega$}$$ where $|x|$ denotes the length (equivalently, the domain) of the word $x$. The notion of a finite rapidly increasing sequences is defined similarly. We equip $W_{Lv_k}^{[\infty]}$ with the \emph{metrizable topology}, i.e. the Polish topology generated by sets of the form $$\{(y_n)_{n<\omega} \in W_{Lv_k}^{[\infty]} : y_i = x_i \text{ for all $i \leq m$}\}$$ where $(x_0, \dots, x_m)$ is a finite rapidly increasing sequence in $W_{Lv_k}$. Given $x \in [X]_{Lv_k}$, the \emph{support} of $x$ in $X$, denoted $\supp_X(x)$, is the set $\{n_0 < \dots < n_m\}$ of indices such that $$T^{j_0}(x_{n_0}[\vec\lambda_0])^\frown \dots ^\frown T^{j_l}(x_{n_l}[\vec\lambda_l])$$ for some choice of $n_0 < \dots < n_l < \omega$, $j_0, \dots, j_l \leq k$ and $\vec\lambda_i \in L_{n_i}^k \cup \{\vec v\}$. The requirement that our sequences be rapidly increasing is necessary to ensure that $\supp_X(x)$ is uniquely defined. Using this observation, we can define an ordering $\leq$ on $W_{Lv_k}^{[\infty]}$ by setting, for rapidly increasing sequences $X = (x_n)$ and $Y$, $X \leq Y$ if and only if $x_n \in [Y]_{Lv_k}$ for all $n < \omega$ and $$\max \supp_Y(x_n) < \min \supp_Y(x_m) \text{ whenever $n < m$}.$$ In this case, we say that $X$ is a \emph{block subsequence} of $Y$; we denote by $[Y]_{Lv_k}^{[\infty]}$ the set of all infinite block subsequences of $Y$. The set of all finite block subsequences of a a finite rapidly increasing sequence $(y_0, \dots, y_m)$ will be denoted by $[y_0,\dots,y_m]_{Lv_k}$.

Our first goal is to prove the following theorem, which is a common generalization of Gowers' $\fin_k$ theorem and the infinitary Hales-Jewett theorem.

\begin{thm}\label{infGHJ}
For every finite Souslin measurable colouring of $W_{Lv_k}^{[\infty]}$, there is $X \in W_{Lv_k}^{[\infty]}$ such that $[X]_{Lv_k}^{[\infty]}$ is monochromatic.
\end{thm}

To prove such a result we will use \emph{ultra-Ramsey theory} as developed in \cite{T}. Before we describe the relevant results, we need to construct an ultrafilter on $W_{Lv_k}$ which will be used throughout this section. To this end, work in the Stone-\v{C}ech compactification $(\beta S, \, ^\frown)$ of the semigroup $(S, \, ^\frown)$; we view $\beta S$ as the compact Hausdorff space consisting of all ultrafilters on $S$ with the topology generated by basic open sets of the form $$\overline{A} := \{\U \in \beta S : A \in \U\}$$ where $A$ is a non-empty subset of $S$. Given $\U \in \beta S$ and a first-order formula $\varphi(x)$ with a free variable $x$ ranging over elements of $S$, write $$(\U x) \varphi(x) \iff \{x \in S : \varphi(x)\} \in \U.$$ In this way, each $\U \in \beta S$ corresponds to an \emph{ultrafilter quantifier} on $S$. It is easy to check that ultrafilter quantifiers commute with conjunction and negation of first-order formulas. Using ultrafilter quantifiers, the extension of the concatenation operation to $\beta S$ is characterized as follows: $$A \in \U ^\frown \V \iff (\U x)(\V y) \, \, x^\frown y \in A.$$ Similarly, the extension of the tetris operation to $\beta S$ is determined by $$A \in T(\U) \iff (\U x) \, \, T(x) \in A.$$

Let $S^*$ denote the closed subsemigroup of $\beta S$ consisting of all non-principal ultrafilters on $S$ which are \emph{cofinite}, i.e. ultrafilters $\U$ such that $$\{x \in W_{Lv_i} : |x| > n\} \in \U \text{ for all $n < \omega$}.$$ Define $$S^*_L = \{\U \in S^* : W_L \in \U\}$$ and, for each $i \in \{1,\dots, k\}$, $$S^*_{Lv_i} = \{ \U \in S^* : W_{Lv_i} \in \U\}.$$ Then $S^*_L$ and $S^*_{Lv_i}$ (for each $1\leq i \leq k$) are closed subsemigroups of $S^*$. Let $\W$ be a minimal idempotent in $S^*_L$, and choose any idempotent $\V_1 \leq \W$ in $S^*_{Lv_1}$. Starting with $\V_1$, recursively construct a sequence $(\V_i)_{1 \leq i \leq k}$ of idempotents such that for each $i<j$:
\begin{enumerate}
	\item $\V_i$ is an idempotent in $S^*_{Lv_i}$.
	\item $\V_i \geq \V_j$.
	\item $T^{(j-i)}(\V_j) = \V_i$.
\end{enumerate}
Assume $\V_1, \dots, \V_{i-1}$ have been constructed and let $$S_i = \{\U \in S^*_{Lv_i} : T(\U) = \V_{i-1}\}.$$ Since $T : S^*_{Lv_i} \rightarrow S^*_{Lv_{i-1}}$ is a continuous surjective homomorphism, it follows (as in the proof of \cite[Lemma 2.24]{T}) that $$S_i ^\frown \V_{i-1} = \{\U ^\frown \V_{i-1} : \U \in S_i\}$$ is a non-empty closed subsemigroup of $S_i$. Thus there is an idempotent in ${S_i} ^\frown \V_{i-1}$ of the form $\U ^\frown \V_{i-1}$; then let $$\V_i = {\V_{i-1}} ^\frown \U ^\frown V_{i-1}.$$ It is routine to check that $\V_i$ is an idempotent which satisfies the required properties.

We will also need the following:

\begin{claim}\label{claim1}
For each $\vec\lambda \in L^k$ and each $i \leq k$, $\V_i[\vec\lambda] = \W$.
\end{claim}
\begin{proof}
Since each mapping $\U \mapsto \U[\vec\lambda]$ for $\vec \lambda \in L^k$ is a homomorphism, it follows that $\V_i[\vec\lambda] \in S^*_L$ is an idempotent and $\V_i[\vec\lambda] \leq \W[\vec\lambda] = \W$, so that $\V_i[\vec\lambda] = \W$ by minimality of $\W$.
\end{proof}

Let $W_{Lv_k}^{[<\infty]}$ be the set of all finite rapidly increasing sequences in $W_{Lv_k}$. We view $W_{Lv_k}^{[<\infty]}$ as a tree ordered by end-extension and with root $\emptyset$, the empty word. The next two definitions are adapted from \cite[Chapter 7.2]{T} by replacing the the tree $\N^{[<\infty]}$ of finite subsets of $\N$ with $W_{Lv_k}^{[<\infty]}$.

\begin{defi}
A \emph{$\V_k$-tree} is a downward closed subtree $U \subseteq W_{Lv_k}^{[<\infty]}$ such that $$U_t := \{x \in W_{Lv_k} : (t, x) \in U\} \in \V_k$$ for all $t \in U$ which extend the stem of $U$, where the \emph{stem} is the $\sqsubseteq$-maximal element of $U$ which is comparable to every other node of the tree. The stem of a $\V_k$-tree $U$ will be denoted by $\stem(U)$.
\end{defi}

Given a $\V_k$-tree $U$, the set of infinite branches of $U$ is denoted by $$[U] := \{(x_n)_{n < \omega} \in W_{Lv_k}^{[\infty]} : (x_0, \dots, x_m) \in U \text{ for all $m < \omega$}\}.$$ For $t \in U$ let $|t|$ denote the \emph{length} of $t$, which is just the domain of $t$ when viewed as a finite sequence in $W_{Lv_k}^{[<\infty]}$.

\begin{defi}
Let $\X \subseteq W_{Lv_k}^{[\infty]}$. $\X$ is \emph{$\V_k$-open} if for every $A \in \X$ there is a $\V_k$-tree $U$ such that $A \in [U] \subseteq \X$. $\X$ is \emph{$\V_k$-Ramsey} if for every $\V_k$-tree $U$ there is a $\V_k$-subtree $U' \subseteq U$ with $\stem(U) = \stem(U')$ such that $[U'] \subseteq \X$ or $[U'] \subseteq \X^c$.
\end{defi}

The collection of all $\V_k$-open subsets of $W_{Lv_k}^{[\infty]}$ forms a topology, called the \emph{$\V_k$-topology}, which refines the metrizable topology of $W_{Lv_k}^{[\infty]}$. The next two results are adapted from \cite[Chapter 7.2]{T} by replacing the tree $\N^{[<\infty]}$ of finite subsets of $\N$ ordered by end-extension with the tree $W_{Lv_k}^{[<\infty]}$. We then have the following version of Todorcevic's ultra-Ellentuck theorem from \cite[Chapter 7]{T}.

\begin{thm}
Let $\X \subseteq W_{Lv_k}^{[\infty]}$. Then $\X$ has the property of Baire relative to the $\V_k$-topology if and only if $\X$ is $\V_k$-Ramsey.
\end{thm}

Using the fact that the property of Baire is preserved under the Souslin operation (see, e.g., \cite[Corollary 4.8]{T}) we then have:

\begin{thm}\label{cor1}
For every $r \in \N$ and every Souslin measurable $c : W_{Lv_k}^{[\infty]} \rightarrow r$ there are $i < r$ and a $\V_k$-tree $U$ with stem $\emptyset$ such that $[U] \subseteq c^{-1}\{i\}$.
\end{thm}

Our next goal is to show that for any $\V_k$-tree $U$ there is a rapidly increasing word $Y$ with the property that $X \in [U]$ whenever $X \leq Y$. To this end, we have the following key lemma:

\begin{lem}\label{lem1}
For every $\V_k$-tree $U$ with stem $\emptyset$ there is $Y = (y_n)_{n < \omega} \in W_{Lv_k}^{[\infty]}$ together with a decreasing sequence $(A_n)_{n<\omega}$ of subsets of $W_{Lv_k}$ such that:
	\begin{enumerate}[(a)]
		\item $A_n \subseteq U_t$ for every $t \in U$ such that $t \leq (y_0, \dots, y_{n-1})$.
		\item $[ y_m, \dots, y_n]_{Lv_k} \subseteq A_m$ for all $m \leq n < \omega$.
	\end{enumerate}
\end{lem}
\begin{proof}
By induction on $n$, define a decreasing sequence $(A_n)_{n<\omega}$ together with a rapidly increasing sequence $(y_n)_{n<\omega}$ such that, for all $n < \omega$:
	\begin{enumerate}
		\item $y_n \in A_n \in \V_k$.
		\item $A_{n+1} \subseteq \{z \in W_{Lv_k} : [ y_n, z ]_{Lv_k}  \subseteq A_n\}$.
		\item $A_n \subseteq U_t$ for every $t \in U$ such that $t \leq (y_i)_{i < n}$.
	\end{enumerate}
To start, take $A_0 := U_\emptyset$ and note that $A_0 \in \V_k$ since $U$ is $\V_k$-tree. Using the properties of the sequence $(\V_i)_{1 \leq i \leq k}$ of idempotents constructed above, we have $$\V_k = T^j(\V_k) ^\frown \V_k = {\V_k}^\frown T^j(\V_k)$$ for all $j \leq k$. Rewriting this fact in terms of the ultrafilter quantifier and using the fact that $A_0 \in V_k$, it follows that $$(\V_k y)(\V_k z) \left([ y, z ]_{Lv_k}  \subseteq A_0\right)$$ and so we take any $y_0 \in W_{Lv_k}$ such that $(\V_k z)\left( [ y_0, z ]_{Lv_k}  \subseteq A_0\right)$; in particular $y_0 \in A_0$ by definition of $[ y_0, z ]_{Lv_k} $. We then take $A_1$ to be the intersection of the set $\{z \in W_{Lv_k} : [ y_0, z ]_{Lv_k}  \subseteq A_0\}$ with $$\bigcap\left\{ U_t : t\in U \text{ and} \, t \leq (y_0) \right\}.$$ Note that $A_0 \supseteq A_1$ and $A_1 \in \V_k$ since there are only finitely many $t \in U$ satisfying $t \leq (y_0)$ and since each $U_t \in \V_k$.

Now suppose $A_0, \dots, A_n$ and $y_0, \dots, y_{n-1}$ have been constructed. Since $\V_k$ is cofinite, it follows that there is $y_n \in W_{Lv_k}$ such that $$ |y_n| > \sum_{i=0}^{n-1} |y_i|$$ and $(\V_k z)\left( [ y_n, z ]_{Lv_k}  \subseteq A_n\right)$; in particular $y_n \in A_n$. Then take $A_{n+1}$ to be the intersection of the set $\{z \in W_{Lv_k} : [ y_n, z ]_{Lv_k}  \subseteq A_n\}$ with $$\bigcap\left\{ U_t  :t \in U \text{ and} \, t \leq (y_0,\dots, y_n) \right\}.$$ Observe that the collection $[y_0,\dots, y_{n-1}]_{Lv_k} $ is finite since we only allow substitutions of the form $y_i[\vec \lambda_i]$ for $\vec \lambda_i \in L_i^k \cup \{\vec v\}$ and so there are only finitely many sets in the above intersection. Thus $A_{n+1} \in \V_k$ and $A_n \supseteq A_{n+1}$. This completes the inductive construction of the sequences $(A_n)$ and $(y_n)$. In particular, condition (a) is satisfied by (3).

We check condition (b) by downward induction on $m \leq n$ for $n < \omega$ fixed. The case $m = n$ follows from (1), while the case $m = n -1 $ follows using (1) and (2) to obtain $[ y_{n-1}, y_n ]_{Lv_k}  \subseteq A_{n-1}$. Now suppose inductively that (b) holds for some $m \leq n$; we aim to show $[ y_{m-1}, y_m, \dots, y_n ]_{Lv_k}  \subseteq A_{m-1}$. Take any $$z = T^{j_{m-1}}(y_{m-1}[\vec \lambda_{m-1}]) ^\frown \dots ^\frown T^{j_n}(y_n[\vec \lambda_n])$$ with $j_{m-1}, \dots, j_n \leq k$ and $\vec \lambda_i \in L_i^k \cup \{\vec v\}$ are such that $\min j_i = 0$ and $\vec \lambda_i = \vec v$ for some $i \in \{m-1, \dots, n\}$. We consider two cases: Suppose first that there is $i > m-1$ such that $j_i = 0$. Then $$z' := T^{j_m}(y_m[\vec \lambda_m]) ^\frown \dots ^\frown T^{j_n}(y_n[\vec \lambda_n]) \in [ y_m, \dots, y_n]_{Lv_k}  \subseteq A_m$$ where the inclusion comes from the inductive hypothesis. Thus $z' \in A_m$ and so $$z \in [ y_{m-1}, z ]_{Lv_k}  \subseteq A_{m-1}$$ by (2). Now suppose $j_i > 0$ for each $i > m-1$ (so that, in particular, $j_{m-1} = 0$). Let $l := \min\{j_m, \dots, j_n\} > 0$ and write $$z = {y_{m-1}}^\frown  T^l \left ( T^{j_m - l}(y_m[\vec \lambda_m]) ^\frown \dots ^\frown T^{j_n - l}(y_n[\vec \lambda_n])\right).$$ By the inductive hypothesis we have $$z'' := T^{j_m - l}(y_m[\vec \lambda_m]) ^\frown \dots ^\frown T^{j_n - l}(y_n[\vec \lambda_n]) \in [ y_m, \dots, y_n ]_{Lv_k}  \subseteq A_m,$$ and so $z \in [ y_{m-1}, z'' ]_{Lv_k}  \subseteq A_{m-1}$ by (2). This completes the proof of the lemma.
\end{proof}

\begin{proof}[Proof of Theorem \ref{infGHJ}]
Let $c : W^{[\infty]}_{Lv_k} \rightarrow r$ be Souslin measurable. By Corollary \ref{cor1} there is a $\V_k$-tree $U$ with stem $\emptyset$ such that $[U] \subseteq c^{-1}\{i\}$ for some $i < r$. Let $Y = (y_n)$ be the rapidly increasing sequence given by applying Lemma \ref{lem1} to $U$. To finish the proof of the theorem, it is enough to show $[Y]_{Lv_k}^{[\infty]} \subseteq [U]$. Let $X = (x_n) \leq Y$; we show $X \in [U]$ by induction on the length $m$ of $s := (x_0, \dots, x_{m-1})$. When $m=0$ we have $s = \emptyset$ which belongs to $U$ by assumption. So we assume $s \in U$ and show $(s, x_m) \in U$. Since $x_m \in [Y]_{Lv_k}$, we can write $$x = T^{j_0}(y_{n_0}[\vec \lambda_0]) ^\frown \dots ^\frown T^{j_l}(y_{n_l}[\vec \lambda_l])$$ for some $n_0 < \dots < n_l < \omega, j_0, \dots, j_l \leq k$ and $\vec \lambda_i \in L_{n_i}^k \cup \{\vec v\}$ such that $j_i = k$ and $\vec \lambda_i = \vec v$ for some $i \leq l$, i.e. $x_m \in [y_{n_0}, \dots, y_{n_l}]_{Lv_k}$. By definition of $Y$, $$[y_{n_0}, \dots, y_{n_l}]_{Lv_k} \subseteq A_{n_0} \subseteq U_t$$ for each $t \in U$ such that $t \leq (y_0, \dots, y_{n-1})$. Since $X$ is a block subsequence of $Y$, we have $\max \supp_Y(x_{m-1}) < \min \supp_Y(x_m)$ and so $s \leq (y_0, \dots, y_{n-1})$. Thus $x_m \in U_s$ and so $(s, x_m) \in U$, as required.
\end{proof}

We are now in a position to prove the main theorem of this section, which allows us to parametrize the Milliken-Todorcevic theorem by a sequence of perfect subsets of $2^\omega$. Parts of the proof are similar to that of \cite[Theorem 5.45]{T}, but we include the details for the sake of completeness.

\begin{proof}[Proof of Theorem \ref{parafin}] Fix a finite Souslin measurable colouring $c$ of the product $\fin_k^{[\infty]} \times (2^\omega)^\omega$. Let $$L_n = \{ \sigma \in 2^\omega : (\forall i > n) \, \sigma(i) = 0\}$$ and $L = \bigcup_{n<\omega} L_n$. Define a mapping $\varphi : W^{[\infty]}_{Lv_k} \rightarrow \fin_k^{[\infty]}$ as follows: Given $(x_m)_{m<\omega} \in W^{[\infty]}_{Lv_k}$, let $\varphi((x_m)) = (a_m)$, where $a_m \in \fin_k$ consists of all ordered pairs of the form $$\langle |x_0| + \dots + |x_{m-1}| + l, i \rangle$$ where $v_i$ occupies the $l^{\mathrm{th}}$ place in $x_m$, and where $a_m$ takes the value 0 at all other points of $\omega$. We also define a mapping $\psi : W_{Lv_k}^{[\infty]} \rightarrow 2^{\omega \times \omega}$ by $$\psi((x_m))(n,i) = \sigma(i)$$ if $\sigma \in L$ occupies the $n^{\mathrm{th}}$ place in the infinite variable word $${x_0}^\frown {x_1} ^\frown {x_2} ^\frown \dots$$ and where $\psi((x_m))(n,i) = 0$ if a variable occupies the $n^{\mathrm{th}}$ place in the above infinite word.

Define a colouring $c^*$ of $W_{Lv_k}^{[\infty]}$ by setting $$c^*((x_m)) = c\left( \varphi((x_m)), \psi((x_m))\right)$$ where $(2^\omega)^\omega$ and $2^{\omega \times \omega}$ are identified via the mapping $$(\varep_{n,i})_n)_i \mapsto (\varep_{n,i})_{(n,i)}.$$ It is easy to check that $\varphi$ and $\psi$ are both continuous, from which it follows that $c^*$ is Souslin measurable. Apply Theorem \ref{infGHJ} to find $Y = (y_m) \in W_{Lv_k}^{[\infty]}$ such that $[Y]_{Lv_k}^{[\infty]}$ is monochromatic for $c^*$. Using $Y$, we define a block sequence $B = (b_m) \in \fin_k^{[\infty]}$ where $b_m$ consists of all ordered pairs of the form $$\langle |y_0| + \dots + |y_{2m}| + l, i \rangle$$ where the $l^{\mathrm{th}}$ place of $y_{2m+1}$ is occupied by $v_i$, and where $b_m$ takes the value 0 at all other points of $\omega$. Let $P$ be the collection of all doubly-indexed sequences $(\varep_{n,i})$ such that $$(\varep_{n,i}) = \psi((y_{2m}[\sigma_{2m}]^\frown y_{2m+1}))$$ for some sequence of letters $(\sigma_{2m}) \in \prod_{m < \omega} L_{2m}$. Note that $P$ is contained in the image of $[Y]_{Lv_k}^{[\infty]}$ under $\psi$.

The proof of Theorem \ref{parafin} will be complete once we prove the following two claims:

\begin{claim}\label{claim}
There is an infinite sequence $(P_i)$ of perfect subsets of $2^\omega$ such that $\prod_{i<\omega} P_i \subseteq P$.
\end{claim}
\begin{proof}
Let $y$ denote the infinite variable word $$y_0 \, ^\frown y_1 \, ^\frown y_2 \, ^\frown \dots$$ and, for each $m > 0$, let $I_{2m-1}$ be the interval $$[|y_0| + \dots + |y_{2m-1}|, |y_0| + \dots + |y_{2m-1}| + |y_{2m}|).$$ For each $i < \omega$, let $P_i$ be the set of all $\delta \in 2^\omega$ satisfying the following conditions:
	\begin{enumerate}[(1)]
	\item If $y(n) \in L$, then $\delta(n) = y(n)(i)$.
	\item If $n < |y_0| + \dots + |y_{2i-1}|$ and $y(n)$ is a variable, then $\delta(n) = 0$.
	\item $\delta(n) = \delta(n')$ for all $n, n' \in I_{2m-1}$ such that $y(n)$ and $y(n')$ are variables.
	\end{enumerate}
Since $P_i$ has no restrictions at the minimal place of each interval $I_{2m-1}$ where a variable occurs, it follows that $P_i$ is perfect. To show the required inclusion of sets, let $(\delta_i) \in \prod_{i<\omega} P_i$ and let $(\varep_{n,i})$ be the doubly-indexed sequence such that $\varep_{n,i} = \delta_i(n)$. For each $m$, let $n_m$ be the least place in the interval $I_{2m-1}$ where a variable occurs in $y$. Then for each $m$ choose $\sigma_{2m} \in L_{2m}$ such that $\sigma_{2m}(i) = \delta_i(n_m)$ for each $i < \omega$. Then it is routine to check that the sequence $(\sigma_{2m})$ witnesses the fact that $(\delta_i) \in P$. This proves the claim.
\end{proof}

\begin{claim}
$[B]_k^{[\infty]} \times P \subseteq (\varphi \times \psi)[Y]_{Lv_k}^{[\infty]}$.
\end{claim}
\begin{proof}
Let $(A, (\varep_{n,i})) \in [B]_k^{[\infty]} \times P$. By definition of $P$, there is a sequence $(\sigma_{2m}) \in \prod_{m <\omega} L_{2m}$ such that $$(\varep_{n,i}) = \psi((y_{2m}[\sigma_{2m}]^\frown y_{2m+1})).$$ If we let $X = (x_m)_{m<\omega}$ be given by $x_m = y_{2m}[\sigma_m]^\frown y_{2m+1}$ then $\varphi(X) = B$. For each $l < \omega$, let $I_l$ be the smallest interval of integers such that $$a_l = \sum_{i \in I_l} T^{j_i}(b_i)$$ for some integers $j_i \leq k$, and note that the sequence $(I_l)_{l < \omega}$ is a block sequence. Fix $l < \omega$ and let $\{p, p+1, \dots, p+q\}$ be an enumeration of the interval $$(\max(I_{l-1}), \max(I_l)]$$ where we set $\max(I_{-1}) = -1$ for convenience. Then let $$z_l = T^{r_0}(x_p[\vec \lambda_0])^\frown \dots ^\frown T^{r_q}(x_{p+q}[\vec \lambda_q])$$ where the parameters are determined as follows:
	\begin{enumerate}[(i)]
	\item If $p+i \not \in I_l$, then let $\vec \lambda_i$ be the $k$-tuple $(\vec 0, \dots, \vec 0)$ where $\vec 0 \in 2^\omega$ is the sequence which is constantly 0. In this case, let $r_i = 0$.
	\item If $p+i \in I_l$, then let $\vec \lambda_i = \vec v$ and $r_i = j_{p+i}$.
	\end{enumerate}
Then $Z = (z_l)_{l <\omega}$ is a block subsequence of $X$ and hence of $Y$. By construction, $\varphi(Z) = A$. Finally, note that $\psi(Z) = (\varep_{n,i})$ since the infinite word $$z_0 \, ^\frown z_1 \ ^\frown z_2 \, ^\frown \dots$$ is obtained from the infinite word $$y_0[\sigma_0]^\frown y_1 \, ^\frown \dots ^\frown y_{2m}[\sigma_{2m}]^\frown y_{2m+1}^\frown \dots$$ by replacing some occurrences of a variable with the constant sequence $\vec 0 \in L$. In particular, this shows $\psi(Z) = \psi((y_{2m}[\sigma_{2m}]^\frown y_{2m+1}))$. Thus $$(A, (\varep_{n,i})) = (\varphi(Z), \psi(Z))$$ as required.
\end{proof}

This finishes the proofs of the two claims, and hence the proof of the theorem is complete.
\end{proof}

\section{A parametrized $\fin_{\pm k}^{[\infty]}$ theorem}

In this section we prove the following approximate Ramsey theorem, which parametrizes the infinite-dimensional version of Gowers' $\fin_{\pm k}$ theorem from \cite{JK}. First, given two infinite block sequences $A = (a_n)$ and $B = (b_n)$ in $\fin_{\pm k}$, let $$|| A - B || = \sup_{n<\omega} ||a_n - b_n||_\infty.$$

\begin{thm}[Parametrized $\fin_{\pm k}^{[\infty]}$ Theorem]\label{parafin1}
For every finite Souslin measurable colouring $c : \fin_{\pm k}^{[\infty]} \times (2^\omega)^\omega \rightarrow n$, there are $B \in \fin_{\pm k}^{[\infty]}$, a sequence $(P_i)_{i<\omega}$ of non-empty perfect subsets of $2^\omega$, and $j < n$ such that the following holds: For every $(A, (p_i)) \in [B]_{\pm k}^{[\infty]} \times \prod_{i<\omega}^\infty P_i$ there is $\widetilde{A} \in \fin_{\pm k}^{[\infty]}$ such that $$c(\widetilde{A}, (p_i)) = j \,\, \text{ and } \,\, ||A - \widetilde{A}|| \leq 1.$$
\end{thm}

To prove this result, we will need to develop an infinite-dimensional version of the Gowers-Hales-Jewett theorem which can code information about $\fin_{\pm k}$. As before, fix an infinite alphabet $L = \bigcup_{n<\omega} L_n$ given as an increasing union of finite subalphabets $L_n$, as well as a distinguished letter $0 \in L_0$ together with variables $$v_1, v_{-1}, v_2, v_{-2}, \dots, v_k, v_{-k} \not \in L.$$ $W_L$ will denote the set of all variable-free words over $L$ and, for each $i \in \{1, \dots, k\}$, $W_{Lv_{\pm i}}$ will denote the set of all variable words $x$ over $L$ such that $$ i = \max\{j \leq k : \text{ $v_j$ or $v_{-j}$ appears in $x$}\}.$$ Let $$S = W_L \cup \bigcup_{1\leq i \leq k} W_{Lv_{\pm i}}$$ and work in the semigroup $(S, \, ^\frown)$. Given $x \in W_{Lv_{\pm k}}$ and a $2k$-tuple $$\vec\lambda = (\lambda_{-k}, \dots, \lambda_{-1}, \lambda_1, \dots, \lambda_k) \in L^{2k} \cup \{\vec v \},$$ let $x[\vec\lambda]$ be the word obtained by replacing each occurrence of $v_j$ with $\lambda_j$ for each $j \in \{\pm 1, \dots, \pm k\}$, where $$\vec v = (v_{-k}, \dots, v_{-1}, v_1, \dots, v_k).$$ The tetris operation $T : W_{Lv_{\pm k}} \rightarrow W_{Lv_{\pm (k-1)}}$ is defined as follows: Given $x \in W_{Lv_{\pm k}}$, define $T(x) \in W_{Lv_{\pm (k-1)}}$ by
\[ T(x)(n) := \begin{cases}
	v_{i-1} & \text{ if $x(n) = v_i$ for $i>1$},\\
	v_{i+1} & \text{ if $x(n) = v_i$ for $i < -1$},\\
	0 & \text{ if $x(n) \in \{v_1, v_{-1}\}$},\\
	x(n) & \text{ if $x(n) \in L$}.
	\end{cases}
\]
As before, set $T(w) = w$ for each $w \in W_L$. In this setting we also have a notion of reflection: Given $x \in S$, let $-x$ be the word obtained by replacing each occurrence of a variable $v_i$ with $v_{-i}$ for each $i \in \{\pm 1, \dots, \pm k\}$. Note that the mapping $x \mapsto -x$ is a semigroup homomorphism which is equal to the identity when restricted to $W_L$.

Given a sequence $X = (x_n)_{n < \omega}$ in $W_{Lv_{\pm k}}$, the \emph{partial subsemigroup of $W_{Lv_{\pm k}}$ generated by $X$}, denoted $[X]_{Lv_{\pm k}}$, is defined to be the set of all elements of $W_{Lv_{\pm k}}$ which are of the form $$\varep_0 T^{j_0}(x_{n_0}[\vec\lambda_0])^\frown \dots ^\frown \varep_l T^{j_l}(x_{n_l}[\vec\lambda_l])$$ where $l < \omega$, $n_0 < \dots < n_l, j_i \leq k, \varep_i \in \{\pm 1\}$ and $\vec\lambda_i \in L_{n_i}^{2k} \cup \{ \vec v\}$ for each $i \leq k$.

Let $W_{Lv_{\pm k}}^{[\infty]}$ denote the set of all \emph{rapidly increasing} sequences in $W_{Lv_{\pm k}}$, defined as in the previous section and equipped with its natural metrizable topology. Exactly as before, the notion of rapidly increasing allows us to uniquely define the \emph{support} of a word $x \in [X]_{Lv_{\pm k}}$ relative to some rapidly increasing sequence $X$. Given $X = (x_n)_{n<\omega}$ and $Y \in W_{Lv_{\pm k}}^{[\infty]}$, write $X \leq Y$ if and only if $x_n \in [Y]_{Lv_{\pm k}}$ for all $n < \omega$ and $$\max \supp_Y(x_n) < \min \supp_Y(x_m) \text{ whenever $n < m$}.$$ As before, when this happens we say that $X$ is a \emph{block subsequence} of $Y$ and we write $[Y]_{Lv_{\pm k}}^{[\infty]}$ for the set of all infinite block subsequences of $Y$. As is the case for $\fin_{\pm k}$, we cannot expect to obtain an exact Ramsey theorem in this setting; rather, we will only be able to prove an \emph{approximate} version of such a theorem which will make use of a suitable metric. First, we need the following:

\begin{defi}
For a word $x \in W_{Lv_{\pm k}}$, define $$L(x) = \{n < |x| : x(n) \in L \setminus \{0\}\}.$$ Two words $x, y \in W_{Lv_{\pm k}}$ are \emph{compatible} if:
	\begin{enumerate}[(i)]
	\item $|x| = |y|$.
	\item $L(x) = L(y)$ and $x(n) = y(n)$ for all $n \in L(x)$.
	\end{enumerate}
\end{defi}

Note that compatibility is a transitive relation on the set of pairs of words. Now, define a metric on the set $\{v_{\pm 1}, \dots, v_{\pm k}\} \cup \{0\}$ by setting $d(v_i, v_j) = |i - j|$ for variables $v_i$ and $v_j$, and $d(v_i, 0) = |i|$. Using this, define a metric $d$ on $W_{Lv_{\pm k}}$ taking values in $\mathbb{R} \cup \{\infty\}$ by
\[ d(x, y) = \begin{cases}
	 \sup\{d(x(i), y(i)) : i \in |x| \setminus L(x)\} & \text{ if $x$ and $y$ are compatible}, \\
	 \infty & \text{ otherwise}.
	 \end{cases}
\]
We then extend this to a metric on $W_{Lv_{\pm k}}^{[\infty]}$, also denoted $d$, by setting $$d((x_n),(y_n)) = \sup_{n<\omega} d(x_n, y_n).$$ For $\varep > 0$, $A \subseteq W_{Lv_{\pm k}}$ and $\X \subseteq W_{Lv_{\pm k}}^{[\infty]}$, let $$(A)_\varep = \{x \in W_{Lv_{\pm k}} : (\exists y \in A) \, d(x,y) \leq \varep\},$$ $$(\X)_\varep = \{X \in W_{Lv_{\pm k}}^{[\infty]} : (\exists Y \in \X) \, d(X,Y) \leq \varep\}.$$

\begin{thm}\label{infapproxGHJ}
For every $k, r \in \N$ and every Souslin measurable $c : W_{Lv_{\pm k}}^{[\infty]} \rightarrow r$ there are $i < r$ and an infinite block sequence $X \in W_{Lv_{\pm k}}^{[\infty]}$ such that $$[X]_{\pm k}^{[\infty]} \subseteq \left(c^{-1}\{i\}\right)_1.$$
\end{thm}

To prove Theorem \ref{infapproxGHJ}, we use ultra-Ramsey theory. First we will construct an ultrafilter which behaves well with respect to the mapping $$-T : W_{Lv_{\pm k}} \rightarrow W_{Lv_{\pm (k-1)}} : x \mapsto -T(x)$$ in a sense that we now make precise. Work in the closed subsemigroup $S^* \subseteq \beta S$ consisting of all non-principal cofinite ultrafilters on $S$, where \emph{cofinite} is defined as before. Define $$S^*_L = \{\U \in S^* : W_L \in \U\}$$ and, for each $i \in \{1,\dots, k\}$, $$S^*_{Lv_{\pm i}} = \{ \U \in S^* : W_{Lv_{\pm i}} \in \U \}.$$ Then $S^*_L$ and $S^*_{Lv_{\pm i}}$ (for each $1\leq i \leq k$) are non-empty closed subsemigroups of $S^*$. Let $\W$ be a minimal idempotent in $S^*_L$, and choose any idempotent $\V_1 \leq \W$ in $S^*_{Lv_1}$. Exactly as in the previous section, recursively construct a sequence $(\V_i)_{1 \leq i \leq k}$ of idempotents starting with $\V_1$ such that for each $i<j$:
\begin{enumerate}
	\item $\V_i$ is an idempotent in $S^*_{Lv_{\pm i}}$.
	\item $\V_i \geq \V_j$.
	\item $(-T)^{(j-i)}(\V_j) = \V_i$.
	\item For each $\vec\lambda \in L^k$, $\V_i[\vec\lambda] = \W$.
\end{enumerate}
In particular, note that (4) implies $(-T)(\V_1) = \W$ since $(-T)(x) = x[\vec 0]$ for each $x \in W_{Lv_{\pm 1}}$ and where $\vec 0 = (0, \dots, 0)$. In addition to the above properties, we will also need the following useful fact. First, given $A \subseteq W_{Lv_{\pm k}}$, let $-A$ be the set of all words of the form $-x$ for $x \in A$.

\begin{lem}
The ultrafilter $\V_k$ is \emph{subsymmetric}, i.e. $-(A)_1 \in \V_k$ whenever $A \in \V_k$.
\end{lem}
\begin{proof}
Since $\V_k \leq \V_{k-1}$ and $(-T)(\V_k) = \V_{k-1}$ by property (3) in the definition of the ultrafilters $(\V_i)_{1\leq i \leq k}$, we have $$\V_k = (-T)(\V_k) ^\frown \V_k = {\V_k}^\frown (-T)\V_k.$$ (When $k=1$, define $\V_0 = \W$.) Thus, for each $A \subseteq W_{Lv_{\pm k}}$,
\begin{equation*}
\begin{split}
A \in \V_k &\iff (\V_k x)(\V_k y) \, (-T)(x)^\frown y \in A \\
& \implies (\V_k x)(\V_k y) \, (-x)^\frown T(y) \in (A)_1 \\
& \iff (\V_k x)(\V_k y) \,  x^\frown (-T)(y) \in -(A)_1 \\
& \iff -(A)_1 \in {\V_k}^\frown (-T)\V_k = \V_k
\end{split}
\end{equation*}
where we use the easy fact that $(-T)(x)^\frown y$ and $(-x)^\frown T(y)$ are compatible.
\end{proof}

View the space $W_{Lv_{\pm k}}^{[< \infty]}$ of finite rapidly increasing sequences as a tree ordered by end-extension and with root $\emptyset$. Fix the subsymmetric cofinite ultrafilter $\V_k$ define above. Exactly as in the previous section, we define the notions of $\V_k$-tree, $\V_k$-open and $\V_k$-Ramsey relative to the tree $W_{Lv_{\pm k}}^{[<\infty]}$. An application of the ultra-Ellentuck theorem in this setting then yields:

\begin{cor}\label{cor2}
For every $r \in \N$ and every Souslin measurable $c : W_{Lv_{\pm k}}^{[\infty]} \rightarrow r$ there are $i < r$ and a $\V_k$-tree $U$ with stem $\emptyset$ such that $[U] \subseteq c^{-1}\{i\}$.
\end{cor}

Given $\alpha \leq \omega$ and a sequence $X = (x_n)_{n<\alpha}$ in $W_{Lv_{\pm k}}$, let $[X]_{(-T)}$ denote the set of all words of the form $$(-T)^{j_0}(x_{n_0}[\vec \lambda_0]) ^\frown \dots ^\frown (-T)^{j_l}(x_{n_l}[\vec \lambda_l])$$ where $l \geq 0, n_0 < \dots < n_l < \alpha, \vec \lambda_i \in L_{n_i}^k \cup \{\vec v\}$, and $j_0, \dots, j_l \leq k$ such that $\min j_i = 0$ and $\vec \lambda_i = \vec v$ for some $i \leq l$. When the sequence $X = (x_n)_{n<m}$ is finite, we will often write $[x_0, \dots, x_{m-1}]_{(-T)}$ for the above collection. If $\alpha \leq \omega$ and $X = (x_n)_{n < \alpha}, \, Y$ are rapidly increasing sequences in $W_{Lv_{\pm k}}$, write $X \leq_{(-T)} Y$ whenever $x_n \in [Y]_{(-T)}$ for every $n < \alpha$ and $$\max \supp_Y(x_n) < \min \supp_Y(x_m) \text{ whenever $n < m < \alpha$}.$$

\begin{lem}\label{lem2}
For every $\V_k$-tree $U$ with stem $\emptyset$ there is $Y = (y_n)_{n < \omega} \in W_{Lv_{\pm k}}^{[\infty]}$ together with a decreasing sequence $(A_n)_{n<\omega}$ of subsets of $W_{Lv_{\pm k}}$ such that:
	\begin{enumerate}[(a)]
		\item $A_n \subseteq U_t \cap -(U_t)_1$ for every $t \in U$ such that $t \leq_{(-T)} (y_0, \dots, y_{n-1})$.
		\item $[ y_m, \dots, y_n]_{(-T)} \subseteq A_m$ for all $m \leq n < \omega$.
	\end{enumerate}
\end{lem}
\begin{proof}
By induction on $n$, define a decreasing sequence $(A_n)_{n<\omega}$ together with a rapidly increasing sequence $(y_n)_{n<\omega}$ such that, for all $n < \omega$:
	\begin{enumerate}
		\item $y_n \in A_n \in \V_k$.
		\item $A_{n+1} \subseteq \{z \in W_{Lv_{\pm k}} : [ y_n, z ]_{(-T)} \subseteq A_n\}$.
		\item $A_n \subseteq U_t \cap -(U_t)_1$ for every $t \in U$ such that $t \leq_{(-T)} (y_i)_{i < n}$.
	\end{enumerate}
To start, take $A_0 := U_\emptyset \cap -(U_\emptyset)_1$ and note that $A_0 \in \V_k$ since $\V_k$ is subsymmetric and $U_\emptyset \in \V_k$. The definition of $\V_k$ implies $$(\V_k y)(\V_k z) \left([ y, z ]_{(-T)} \subseteq A_0\right)$$ and so we take any $y_0 \in W_{Lv_{\pm k}}$ such that $(\V_k z)\left( [ y_0, z ]_{(-T)} \subseteq A_0\right)$; in particular $y_0 \in A_0$ by definition of $[ y_0, z ]_{(-T)}$. We then take $A_1$ to be the intersection of the set $\{z \in W_{Lv_{\pm k}} : [ y_0, z ]_{(-T)} \subseteq A_0\}$ with $$\bigcap\left\{ U_t \cap -(U_t)_1 : t\in U \text{ and} \, t \leq_{(-T)} (y_0) \right\}.$$ Note that $A_0 \supseteq A_1$ and $A_1 \in \V_k$ since there are only finitely many $t \in U$ satisfying $t \leq_{(-T)} (y_0)$, and since each $U_t \cap -(U_t)_1 \in \V_k$ using the fact that $\V_k$ is subsymmetric.

Now suppose $A_0, \dots, A_n$ and $y_0, \dots, y_{n-1}$ have been constructed. Since $\V_k$ is cofinite, it follows that there is $y_n \in W_{Lv_{\pm k}}$ such that $$ |y_n| > \sum_{i=0}^{n-1} |y_i|$$ and $(\V_k z)\left( [ y_n, z ]_{(-T)} \subseteq A_n\right)$; in particular $y_n \in A_n$. Then take $A_{n+1}$ to be the intersection of the set $\{z \in W_{Lv_{\pm k}} : [ y_n, z ]_{(-T)} \subseteq A_n\}$ with $$\bigcap\left\{ U_t \cap -(U_t)_1 :t \in U \text{ and} \, t \leq_{(-T)} (y_0,\dots, y_n) \right\}.$$ Observe that the collection $[y_0,\dots, y_{n-1}]_{(-T)}$ is finite since we only allow substitutions of the form $y_i[\vec \lambda_i]$ for $\vec \lambda_i \in L_i^k \cup \{\vec v\}$ and so there are only finitely many sets in the above intersection. Thus $A_{n+1} \in \V_k$ and $A_n \supseteq A_{n+1}$. This completes the inductive construction of the sequences $(A_n)$ and $(y_n)$. In particular, condition (a) is satisfied by (3). The verification of (b) is exactly the same as that of the corresponding condition in the statement of Lemma \ref{lem1} after making the obvious adjustments.
\end{proof}

\begin{lem}\label{lem3}
Let $U$ be a $\V_k$-tree with $\stem(U) = \emptyset$. Then there is an infinite rapidly increasing sequence $Y = (y_n)_{n < \omega}$ in $W_{Lv_{\pm k}}$ such that $[Y]_{\pm k}^{[\infty]} \subseteq ([U])_2$.
\end{lem}

\begin{proof}
Let $Y$ be as in Lemma \ref{lem2}. We claim that $Y$ satisfies the conclusion of the lemma. To see this, fix an infinite rapidly increasing block subsequence $X = (x_n)_{n < \omega}$ of $Y$. We will construct a rapidly increasing sequence $X' = (x_n')_{n <\omega} \in [U] \cap [Y]^{[\infty]}_{(-T)}$ such that $d(x_n, x_n') \leq 2$ for each $n < \omega$. Suppose, for some $n \geq 0$, we have defined $x_0', \dots, x_{n-1}' \in W_{Lv_{\pm k}}$ such that $s := (x_0', \dots, x_{n-1}') \in U$ and $d(x_i, x_i') \leq 2$ for each $i<n$. (In the case where $n = 0$ we simply have $s = \emptyset$.) Write $$x_n =  \varep_0 T^{j_0}(y_{n_0}[\vec\lambda_0])^\frown \dots ^\frown \varep_l T^{j_l}(y_{n_l}[\vec\lambda_l])$$ where $l < \omega$, $n_0 < \dots < n_l, j_i \leq k, \varep_i \in \{\pm 1\}$ and $\vec\lambda_i \in L_{n_i}^{2k} \cup \{ \vec v\}$ are such that $\min j_i = 0$ and $\vec \lambda_i = \vec v$ for some $i \leq l$. We consider the following two cases:

\begin{case}
There is $i \leq l$ such that $j_i = 0, \vec \lambda_i = \vec v$ and $\varep_i = +1$.

\noindent For each $i \leq l$, set $z_i := \varep_i T^{j_i}(y_{n_i}[\vec \lambda_i])$ for convenience. We consider the following two subcases:
\begin{enumerate}[(a)]
	\item $\varep_i = +1$ and $j_i$ is even, or $\varep_i = -1$ and $j_i$ is odd. In either case, set $z_i' := z_i$ and note that $z_i' = (-T)^{j_i}(y_{n_i}[\vec \lambda_i])$.
	\item $\varep_i = +1$ and $j_i$ is odd, or $\varep_i = -1$ and $j_i$ is even. In either case, set $z_i' := T(z_i)$ and note that $z_i' = (-T)^{j_i + 1}(y_{n_i}[\vec \lambda_i])$.
\end{enumerate}
We then set $$x_n' := {z_0'} ^\frown \dots ^\frown z_l'.$$ Note that $x_n'$ is compatible with $x_n$, and $x_n' \in [ y_{n_i} : i \leq l ]_{(-T)}$ by the assumption given by Case 1. Since $d(z_i, z_i') \leq 1$ for all $i \leq l$ we have $d(x_n, x_n') \leq 1$. Furthermore, by the choice of the sequence $Y$ we have $$[ y_{n_i} : i \leq l ]_{(-T)} \subseteq A_{n_0}$$ (using the notation of Lemma \ref{lem2}) and so $x_n' \in U_t$ for every $t \in U$ such that $t \leq_{(-T)} (y_0, \dots, y_{n_0-1})$. In particular, $x_n' \in U_s$ since $$\max \supp_Y(x_{n-1}) < \min \supp_Y(x_n) = n_0$$ and so $(s, x_n') \in U$.
\end{case}

\begin{case} For every $i \leq l$, if $j_i = 0$ and $\vec \lambda_i = \vec v$, then $\varep_i = -1$.

\noindent Apply Case 1 to $-x_n$ to obtain $z \in [ y_{n_i} : i \leq l ]_{(-T)}$ such that $d(-x_n ,z)\leq 1$. By definition of $Y$ we have $$[ y_{n_i} : i \leq l ]_{(-T)} \subseteq A_{n_0}$$ and so $z \in U_t \cap -(U_t)_1$ for every $t \in U$ such that $t \leq_{(-T)} (y_0, \dots, y_{n_0-1})$. As before, this implies $-z \in (U_s)_1$ and so there is $z' \in U_s$ such that $d(-z, z') \leq 1$. Set $x_n' := z'$. Then $x_n'$ is compatible with $x_n$ and $$d(x_n, x_n') \leq d(x_n, -z) + d(-z, z') = d(-x_n, z) + d(-z, z') \leq 2$$ and so $x_n'$ satisfies our requirements.
\end{case}
This completes the inductive construction of $X'$. It is clear from the above construction that $X' \in [U]$ and $d(x_n, x_n') \leq 2$ for all $n < \omega$ and so $X \in ([U])_2$.
\end{proof}

To minimize the ``error'' in the previous result, we will use the following family of mappings: For each $k \in \N$, let $\Phi_k : W_{Lv_{\pm 2k}} \rightarrow W_{Lv_{\pm k}}$ be defined by setting
\[ \Phi_k(x)(n) := \begin{cases} 
      v_{i/2} & \text{ if $x(n) = v_i$ where $i$ is even}, \\
      v_{(i-1)/2} & \text{ if $x(n) = v_i$ where $i$ is odd and positive}, \\
      v_{(i+1)/2} & \text{ if $x(n) = v_i$ where $i$ is odd and negative}, \\
      x(n) & \text{ if $x(n) \in L$}.
   \end{cases}
\]
The following properties of $\Phi_k$ are easy to check:
\begin{enumerate}[(i)]
	\item $\Phi_k$ is a surjective homomorphism of partial semigroups which, in addition, satisfies $\Phi_k(-x) = -\Phi_k(x)$ for every $x \in W_{Lv_{\pm 2k}}$.
	\item For every $x, y \in W_{Lv_{\pm 2k}}$ and every $i, j \leq k$ with $\min\{i,j\} = 0$, $$\Phi_k\left(T^{2i}(x) ^\frown T^{2j}(y)\right) = T^i(\Phi_k(x))^\frown T^j(\Phi_k(y)).$$
	\item For every $x, y \in W_{Lv_{\pm 2k}}$, $d(x,y) \leq 2 \implies d(\Phi_k(x), \Phi_k(y)) \leq 1$. In particular, $\Phi_k$ preserves the compatibility relation between words.
\end{enumerate}
We extend $\Phi$ to $W_{Lv_{\pm 2k}}^{[\infty]}$ by setting $$\Phi((y_n)_{n < \omega}) := (\Phi(y_n))_{n < \omega}.$$ It is straightforward to check that $\Phi$ is continuous with respect to the usual metrizable topologies. Furthermore, note that if $Y$ and $Y'$ are two sequences in $W_{Lv_{\pm 2k}}$ which satisfy $d(Y,Y') \leq 2$, then $d(\Phi(Y),\Phi(Y')) \leq 1$. We are now ready to finish the proof of the approximate Gowers-Hales-Jewett theorem.

\begin{proof}[Proof of Theorem \ref{infapproxGHJ}]
Let $c : W_{Lv_{\pm k}}^{[\infty]} \rightarrow r$ be Souslin measurable and define a colouring $\widetilde{c} : W_{Lv_{\pm 2k}}^{[\infty]} \rightarrow r$ by setting $\widetilde{c} := c \circ \Phi$. Since $\Phi$ is continuous and $c$ is Souslin measurable, it follows that $\widetilde{c}$ is Souslin measurable. By Corollary \ref{cor2} there are $i < r$ and a $\V_k$-tree $U$ with stem $\emptyset$ such that $[U] \subseteq \widetilde{c}^{-1}\{i\}$. Applying Lemma \ref{lem3}, find an infinite rapidly increasing sequence $\widetilde{Y} = (\widetilde{y_n})_{n < \omega}$ in $W_{Lv_{\pm 2k}}$ such that $[\widetilde{Y}]_{\pm 2k}^{[\infty]} \subseteq ([U])_2$.

Let $Y := \Phi(\widetilde{Y}) \in W_{Lv_{\pm k}}^{[\infty]}$ so that $y_n := \Phi(\widetilde{y_n})$ for each $n<\omega$. We claim that $Y$ satisfies $$[Y]_{\pm k}^{[\infty]} \subseteq \left(c^{-1}\{i\}\right)_1.$$ Indeed, if $X = (x_n)_{n < \omega} \in W_{Lv_{\pm k}}^{[\infty]}$ is an infinite rapidly increasing subsequence of $Y$, then for each $n < \omega$ we have $$x_n = \varep_0 T^{j_0}(y_{n_0}[\vec \lambda_0])^\frown \dots ^\frown \varep_l T^{j_l}(y_{n_l}[\vec \lambda_l])$$ for some $\varep_i \in \{\pm 1\}, n_0 < \dots < n_l, \vec \lambda_i \in L_{n_i}^k \cup \{\vec v\}$ and $j_i \leq k$ such that $\min j_i = 0$ and $\vec \lambda_i = \vec v$ for some $i \leq l$. Then properties (i) and (ii) of $\Phi$ listed above imply $x_n = \Phi(\widetilde{x_n})$, where $$\widetilde{x_n} := \varep_0 T^{2j_0}(y_{n_0}[\vec \lambda_0])^\frown \dots ^\frown \varep_l T^{2j_l}(y_{n_l}[\vec \lambda_l]) \in [\widetilde{Y}]_{\pm k}$$ and so, setting $\widetilde{X} := (\widetilde{x_n})_{n<\omega}$, we see that $X = \Phi(\widetilde{X})$. Since $\widetilde{X}$ is a rapidly increasing subsequence of $\widetilde{Y}$, by our choice of $\widetilde{Y}$ we can find $X' \in \widetilde{c}^{-1}\{i\}$ such that $d(\widetilde{X}, X') \leq 2$. Then, as observed above, property (iii) of $\Phi$ implies $$d(\Phi(\widetilde{X}), \Phi(X')) \leq 1.$$ Since $$i = \widetilde{c}(X') = c(\Phi(X'))$$ we obtain $\Phi(X') \in c^{-1}\{i\}$ and so $X \in \left(c^{-1}\{i\}\right)_1$ as required.
\end{proof}

We are now equipped to prove a parametrized version of the infinite-dimensional $\fin_{\pm k}$ theorem.

\begin{proof}[Proof of Theorem \ref{parafin1}]
Fix a finite Souslin measurable colouring $c$ of $\fin_{\pm k}^{[\infty]} \times (2^\omega)^\omega$. As before, let $$L_n = \{ \sigma \in 2^\omega : (\forall i > n) \, \sigma(i) = 0\}$$ and $L = \bigcup_{n<\omega} L_n$. Define a mapping $\varphi : W^{[\infty]}_{Lv_{\pm k}} \rightarrow \fin_{\pm k}^{[\infty]}$ by setting $\varphi((x_m)) = (a_m)$, where $a_m \in \fin_{\pm k}$ consists of all ordered pairs of the form $$\langle |x_0| + \dots + |x_{m-1}| + l, i \rangle$$ where $v_i$ occupies the $l^{\mathrm{th}}$ place in $x_m$, and where $a_m$ takes the value 0 at all other points of $\omega$. We also define a mapping $\psi : W_{Lv_{\pm k}}^{[\infty]} \rightarrow 2^{\omega \times \omega}$ by $$\psi((x_m))(n,i) = \sigma(i)$$ if $\sigma \in L$ occupies the $n^{\mathrm{th}}$ place in the infinite variable word $${x_0}^\frown {x_1} ^\frown {x_2} ^\frown \dots$$ and where $\psi((x_m))(n,i) = 0$ if a variable occupies the $n^{\mathrm{th}}$ place in the above infinite word.

Define a Souslin measurable colouring $c^*$ of $W_{Lv_{\pm k}}^{[\infty]}$ by setting $$c^*((x_m)) = c\left( \varphi((x_m)), \psi((x_m))\right)$$ and apply Theorem \ref{infapproxGHJ} to find $Y = (y_m) \in W_{Lv_{\pm k}}^{[\infty]}$ and a colour $r$ such that $$[Y]_{Lv_{\pm k}}^{[\infty]} \subseteq \left((c^*)^{-1}\{r\}\right)_1.$$ Using $Y$, we define a block sequence $B = (b_m) \in \fin_{\pm k}^{[\infty]}$ where $b_m$ consists of all ordered pairs of the form $$\langle |y_0| + \dots + |y_{2m}| + l, i \rangle$$ where the $l^{\mathrm{th}}$ place of $y_{2m+1}$ is occupied by $v_i$, and where $b_m$ takes the value 0 at all other points of $\omega$. As before, we let $P$ be the collection of all $(\varep_{n,i})$ such that $$(\varep_{n,i}) = \psi((y_{2m}[\sigma_{2m}]^\frown y_{2m+1}))$$ for some sequence $(\sigma_{2m}) \in \prod_{m < \omega} L_{2m}$.

Exactly as in the proof of Theorem \ref{parafin1} we can show there is an infinite sequence $(P_i)$ of perfect subsets of $2^\omega$ such that $\prod_{i<\omega} P_i \subseteq P$. We will also need the following:

\begin{claim}
$[B]_{\pm k}^{[\infty]} \times P \subseteq (\varphi \times \psi)[Y]_{Lv_{\pm k}}^{[\infty]}$.
\end{claim}
\begin{proof}
Let $(A, (\varep_{n,i})) \in [B]_{\pm k}^{[\infty]} \times P$. By definition of $P$, there is a sequence $(\sigma_{2m}) \in \prod_{m <\omega} L_{2m}$ such that $$(\varep_{n,i}) = \psi((y_{2m}[\sigma_{2m}]^\frown y_{2m+1})).$$ If we let $X = (x_m)_{m<\omega}$ be given by $x_m = y_{2m}[\sigma_m]^\frown y_{2m+1}$ then $\varphi(X) = B$. For each $l < \omega$, let $I_l$ be the smallest interval of integers such that $$a_l = \sum_{i \in I_l} s_i T^{j_i}(b_i)$$ for some integers $j_i \leq k$ and $s_i \in \{ \pm 1\}$, and note that the sequence $(I_l)_{l < \omega}$ is a block sequence. Fix $l < \omega$ and let $\{p, p+1, \dots, p+q\}$ be an enumeration of the interval $$(\max(I_{l-1}), \max(I_l)]$$ where we set $\max(I_{-1}) = -1$ for convenience. Then let $$z_l = \rho_0 T^{r_0}(x_p[\vec \lambda_0])^\frown \dots ^\frown \rho_q T_{r_q}(x_{p+q}[\vec \lambda_q])$$ where the parameters are determined as follows:
	\begin{enumerate}[(i)]
	\item If $p+i \not \in I_l$, then let $\vec \lambda_i$ be the $2k$-tuple $(\vec 0, \dots, \vec 0)$ where $\vec 0 \in 2^\omega$ is the sequence which is constantly 0. In this case, let $r_i = 0$ and $\rho_i = 1$.
	\item If $p+i \in I_l$, then let $\vec \lambda_i = \vec v$, $r_i = j_{p+i}$ and $\rho_i = s_{p+i}$.
	\end{enumerate}
Then, exactly as before, one checks that $(A, (\varep_{n,i})) = (\varphi(Z), \psi(Z))$.
\end{proof}

We now verify that $B$ and $(P_i)$ are as desired. To this end, fix $$(A, (\varep_{n,i})) \in [B]_{\pm k}^{[\infty]} \times \prod_{i <\omega} P_i \subseteq [B]_{\pm k}^{[\infty]} \times P$$ and apply the previous claim to find $Z = (z_n) \in [Y]_{Lv_{\pm k}}^{[\infty]}$ such that $$(A, (\varep_{n,i})) = (\varphi(Z), \psi(Z)).$$ By definition of $Y$, there is $\widetilde{Z} = (\widetilde{z}_n) \in W_{Lv_{\pm k}}^{[\infty]}$ such that $d(Z,\widetilde{Z}) \leq 1$ and $c^*(\widetilde{Z}) = r$. Using the definition of the metric $d$ it must be that $z_n$ is compatible with $\widetilde{z}_n$ for each $n$, and so it follows that $\psi(Z) = \psi(\widetilde{Z})$. Furthermore, note that $$||\varphi(Z) - \varphi(\widetilde{Z})|| = d(Z, \widetilde{Z}) \leq 1$$ according to the definitions of $d$ and $\varphi$. Let $\widetilde{A} = \varphi(\widetilde{Z})$; then $||A - \widetilde{A}|| \leq 1$ and $$c(\widetilde{A}, (\varep_{n,i})) = c((\varphi(\widetilde{Z}), \psi(\widetilde{Z}))) = c^*(\widetilde{Z}) = r.$$ This finishes the proof of theorem.
\end{proof}

As an easy consequence, we obtain a parametrized version of Gowers' $\fin_{\pm k}$ theorem:

\begin{cor}\label{paraGowers}
For every finite colouring $c : \fin_{\pm k} \times (2^\omega)^\omega \rightarrow n$, there are $B \in \fin_{\pm k}^{[\infty]}$, a sequence $(P_i)_{i<\omega}$ of non-empty perfect subsets of $2^\omega$, and $j < n$ such that the following holds: For every $(b, (p_i)) \in [B]_{\pm k} \times \prod_{i<\omega}^\infty P_i$ there is $\widetilde{b} \in \fin_{\pm k}$ such that $$c(\widetilde{b}, (p_i)) = j \text{ and }||b - \widetilde{b}||_\infty \leq 1.$$
\end{cor}

We conclude with an application of the previous result to the oscillation stability of uniformly equicontinuous families of real-valued functions on $S_{c_0}$, the unit sphere of the Banach space $c_0$. The following result can be seen as a parametrization of Gowers' $c_0$ theorem. The proof is similar to Gowers' original proof \cite{G}; see also \cite{G1}.

\begin{thm}
Let $\{f_{\sigma} : \sigma \in (2^\omega)^\omega\}$ be a family of functions $S_{c_0} \rightarrow \R$ which is uniformly bounded and uniformly equicontinuous. Then for every $\varep > 0$ there are an infinite-dimensional subspace $X$ of $c_0$ and a sequence $(P_n)_{n<\omega}$ of perfect subsets of $2^\omega$ such that the oscillation of each mapping $f_{\sigma}$ for $\sigma \in \prod_{n<\omega} P_n$ is at most $\varep$ when restricted to $S_X$, the unit sphere of $X$.
\end{thm}
\begin{proof}
Apply uniform equicontinuity to the given $\varep$ to find $\delta > 0$ such that $$|f_\sigma(x) - f_\sigma(y)| \leq \varep/5$$ for all $\sigma \in (2^\omega)^\omega$ and all $x, y \in S_{c_0}$ such that $||x - y||_\infty \leq \delta$. Fix $k$ large enough such that $(1+\delta)^{1-k} < \delta$ and let $\Delta_{\pm k}$ be the subset of $S_{c_0}$ consisting of all finitely-supported vectors with coordinates belonging to the set $$\{\pm (1+\delta)^{i-k} : i = 1, \dots, k\} \cup \{0\}.$$ Let $\varphi : \fin_{\pm k} \rightarrow \Delta_{\pm k}$ be the bijection defined by
\[ \varphi(p)(n) := \begin{cases} 
      (1+\delta)^{i-k} & \text{ if $p(n) = i > 0$}, \\
      0 & \text{ if $p(n) = 0$}, \\
      -(1+\delta)^{i-k} & \text{ if $p(n) = -i < 0$}.
   \end{cases}
\]
Since the family of functions $(f_\sigma)$ is uniformly bounded, there is a partition of $\bigcup_{\sigma \in (2^\omega)^\omega} \range(f_\sigma)$ into finitely many disjoint intervals $I_0, \dots, I_{l-1}$ such that the length of each interval is at most $\varep/5$. Define a colouring $c : \fin_{\pm k} \times (2^\omega)^\omega \rightarrow l$ by setting $$c(p, \sigma) = j \iff f_\sigma(\varphi(p)) \in I_j$$ and find $B = (b_n) \in \fin_{\pm k}^{[\infty]}$, a sequence $(P_i)_{i<\omega}$ of non-empty perfect subsets of $2^\omega$, and $j < l$ satisfying the conclusion of Corollary \ref{paraGowers} with respect to $c$. Using the choice of $k$ together with the implication $$||p - q||_\infty \leq 1 \implies ||\varphi(p) - \varphi(q)||_\infty \leq \delta,$$ it follows from the choice of $B$ and $(P_i)$ that $$|f_\sigma(\varphi(p)) - f_\sigma(\varphi(q))| \leq \frac{3 \varep}{5} \text{ for all $p, q \in [B]$ and all $\sigma \in \prod_{i<\omega} P_i$}.$$ Now let $X$ be the linear span of the set $\{\varphi(b_n) : n < \omega\}$ in $c_0$. Then it is straightforward to check that the set $\{\varphi(b) : b \in [B]\}$ is a $\delta$-net in $S_X$. Using the previous inequality, this implies $$|f_\sigma(x) - f_\sigma(y)| \leq \varep \text{ for all $x, y \in S_X$ and all $\sigma \in \prod_{i < \omega} P_i$}.$$ Thus the oscillation of each function $f_\sigma$ for $\sigma \in \prod_{i<\omega} P_i$ is at most $\varep$ on $S_X$.
\end{proof}

\subsection*{Acknowledgements} The author thanks Jordi L\'opez-Abad and Stevo Todorcevic for many useful discussions related to the subject matter.

\bibliographystyle{abbrv}
\bibliography{bibliography}

\begin{thebibliography}{10}

\bibitem{AT}
S.~A. Argyros and S.~Todorcevic.
\newblock {\em Ramsey methods in analysis}.
\newblock Advanced Courses in Mathematics. CRM Barcelona. Birkh\"{a}user
  Verlag, Basel, 2005.

\bibitem{AvT}
A.~Avil\'{e}s and S.~Todorcevic.
\newblock Finite basis for analytic multiple gaps.
\newblock {\em Publ. Math. Inst. Hautes \'{E}tudes Sci.}, 121:57--79, 2015.

\bibitem{E}
E.~Ellentuck.
\newblock A new proof that analytic sets are {R}amsey.
\newblock {\em J. Symbolic Logic}, 39:163--165, 1974.

\bibitem{GP}
F.~Galvin and K.~Prikry.
\newblock Borel sets and {R}amsey's theorem.
\newblock {\em J. Symbolic Logic}, 38:193--198, 1973.

\bibitem{G}
W.~T. Gowers.
\newblock Lipschitz functions on classical spaces.
\newblock {\em European J. Combin.}, 13(3):141--151, 1992.

\bibitem{G1}
W.~T. Gowers.
\newblock Ramsey methods in {B}anach spaces.
\newblock In {\em Handbook of the geometry of {B}anach spaces, {V}ol. 2}, pages
  1071--1097. North-Holland, Amsterdam, 2003.

\bibitem{H}
N.~Hindman.
\newblock Finite sums from sequences within cells of a partition of {$N$}.
\newblock {\em J. Combinatorial Theory Ser. A}, 17:1--11, 1974.

\bibitem{K}
V.~Kanellopoulos.
\newblock A proof of {W}. {T}. {G}owers' {$c_0$} theorem.
\newblock {\em Proc. Amer. Math. Soc.}, 132(11):3231--3242, 2004.

\bibitem{JK}
J.~K. Kawach.
\newblock An infinite-dimensional version of {G}owers' $\mathrm{FIN}_{\pm k}$
  theorem.
\newblock {\em Proc. Amer. Math. Soc.} (to appear)

\bibitem{L}
M.~Lupini.
\newblock Actions on semigroups and an infinitary {G}owers-{H}ales-{J}ewett
  {R}amsey theorem.
\newblock {\em Trans. Amer. Math. Soc.}, 371(5):3083--3116, 2019.

\bibitem{Mi}
J.~G. Mijares.
\newblock Parametrizing the abstract {E}llentuck theorem.
\newblock {\em Discrete Math.}, 307(2):216--225, 2007.

\bibitem{MN}
J.~G. Mijares and J.~E. Nieto.
\newblock A parametrization of the abstract {R}amsey theorem.
\newblock {\em Divulg. Mat.}, 16(2):259--274, 2008.

\bibitem{Miller}
A.~W. Miller.
\newblock Infinite combinatorics and definability.
\newblock {\em Ann. Pure Appl. Logic}, 41(2):179--203, 1989.

\bibitem{M}
K.~R. Milliken.
\newblock Ramsey's theorem with sums or unions.
\newblock {\em J. Combinatorial Theory Ser. A}, 18:276--290, 1975.

\bibitem{Ojeda}
D.~Ojeda-Aristizabal.
\newblock Finite forms of {G}owers' theorem on the oscillation stability of
  {$C_0$}.
\newblock {\em Combinatorica}, 37(2):143--155, 2017.

\bibitem{Paw}
J.~Pawlikowski.
\newblock Parametrized {E}llentuck theorem.
\newblock {\em Topology Appl.}, 37(1):65--73, 1990.

\bibitem{T}
S.~Todorcevic.
\newblock {\em Introduction to {R}amsey spaces}, volume 174 of {\em Annals of
  Mathematics Studies}.
\newblock Princeton University Press, Princeton, NJ, 2010.

\bibitem{Tyros}
K.~Tyros.
\newblock Primitive recursive bounds for the finite version of {G}owers'
  {$c_0$} theorem.
\newblock {\em Mathematika}, 61(3):501--522, 2015.

\bibitem{Zheng}
Y.~Y. Zheng.
\newblock {\em Parametrizing topological Ramsey spaces}.
\newblock PhD thesis, University of Toronto, 2018.

\end{thebibliography}

\end{document}